\documentclass[leqno]{article}
\usepackage{amsmath,amsthm,bm,colortbl,graphicx,mathrsfs,afterpage,amssymb,url}
\RequirePackage[colorlinks,citecolor=blue,urlcolor=blue]{hyperref}
\pdfoutput=1
\theoremstyle{plain}
\newtheorem{theorem}{Theorem}
\newtheorem{lemma}[theorem]{Lemma}

\theoremstyle{remark}
\newtheorem*{remark}{Remark} 
\def\P{{\rm P}} 
\def\E{{\rm E}} 
\allowdisplaybreaks

\title{Parrondo games with spatial dependence \break and a related spin system, II}
\author{S. N. Ethier\thanks{Partially supported by a grant from the Simons Foundation (209632).  Also supported by a Korean Federation of Science and Technology Societies grant funded by the Korean Government (MEST, Basic Research Promotion Fund).}\\
\begin{small}University of Utah\end{small}\\
\begin{small}Department of Mathematics\end{small}\\
\begin{small}155 South 1400 East, JWB 233\end{small}\\
\begin{small}Salt Lake City, UT 84112, USA\end{small}\\
\begin{small}e-mail: ethier@math.utah.edu\end{small}
\and
Jiyeon Lee\thanks{Supported by the Basic Science Research Program through the National Research Foundation of Korea (NRF) funded by the Ministry of Education, Science and Technology (2012-0004434).}\\
\begin{small}Yeungnam University\end{small}\\
\begin{small}Department of Statistics\end{small}\\
\begin{small}214-1 Daedong, Kyeongsan\end{small}\\
\begin{small}Kyeongbuk 712-749, South Korea\end{small}\\
\begin{small}e-mail: leejy@yu.ac.kr\end{small}}
\date{}

\begin{document}
\maketitle

\begin{abstract}
\noindent Let game $B$ be Toral's cooperative Parrondo game with (one-dimensional) spatial dependence, parameterized by $N\ge3$ and $p_0,p_1,p_2,p_3\in[0,1]$, and let game $A$ be the special case $p_0=p_1=p_2=p_3=1/2$.  Let $\mu_B^N$ (resp., $\mu_{(1/2,1/2)}^N$) denote the mean profit per turn to the ensemble of $N$ players always playing game $B$ (resp., always playing the randomly mixed game $(1/2)(A+B)$).  In previous work we showed that, under certain conditions, both sequences converge and the limits can be expressed in terms of a parameterized spin system on the one-dimensional integer lattice.  Of course one can get similar results for $\mu_{(\gamma,1-\gamma)}^N$ corresponding to $\gamma A+(1-\gamma)B$ for $0<\gamma<1$.  In this paper we replace the random mixture with the nonrandom periodic pattern $A^r B^s$, where $r$ and $s$ are positive integers.  We show that, under certain conditions, $\mu_{[r,s]}^N$, the mean profit per turn to the ensemble of $N$ players repeatedly playing the pattern $A^rB^s$, converges to the same limit that $\mu_{(\gamma,1-\gamma)}^N$ converges to, where $\gamma:=r/(r+s)$.  For a particular choice of the probability parameters, namely $p_0=1$, $p_1=p_2\in(1/2,1)$, and $p_3=0$, we show that the Parrondo effect (i.e., $\mu_B^N\le0$ and $\mu_{[r,s]}^N>0$) is present if and only if $N$ is even, at least when $s=1$.
\medskip

\noindent \textit{AMS 2000 subject classification}: Primary 60K35; secondary 60J20. \smallskip\par
\noindent \textit{Key words and phrases}: Parrondo's paradox, cooperative Parrondo games, discrete-time Markov chain, stationary distribution, strong law of large numbers, interacting particle system, spin system, ergodicity. 
\end{abstract}

\section{Introduction}\label{intro}

In Toral's \cite{T01} \textit{cooperative Parrondo games}, there are $N\ge3$ players labeled from 1 to $N$ and arranged in a circle in clockwise order.  At each turn, one player is chosen at random to play.  Call him player $i$.  He plays either game $A$ or game $B$.  In game $A$ he tosses a fair coin.  In game $B$ he tosses a $p_0$-coin (i.e., $p_0$ is the probability of heads) if his neighbors $i-1$ and $i+1$ are both losers, a $p_1$-coin if $i-1$ is a loser and $i+1$ is a winner, a $p_2$-coin if $i-1$ is a winner and $i+1$ is a loser, and a $p_3$-coin if $i-1$ and $i+1$ are both winners.  (Because of the circular arrangement, player 0 is player $N$ and player $N+1$ is player 1.) A player's status as winner or loser depends on the result of his most recent game.  The player of either game wins one unit with heads and loses one unit with tails.  Under these assumptions, the model has an integer parameter $N\ge3$ and four probability parameters $p_0,p_1,p_2,p_3\in[0,1]$.  Game $A$ is fair, so the games are said to exhibit the \textit{Parrondo effect} if game $B$ is losing or fair and the random mixture $C:=\gamma A+(1-\gamma)B$ (i.e., toss a $\gamma$-coin, playing game $A$ if heads, game $B$ if tails) or the nonrandom periodic pattern $C:=A^r B^s$ is winning.  Toral used simulation to find a case (namely, $N=50$, $100$, or $200$, $p_0=1$, $p_1=p_2=4/25$, and $p_3=7/10$) in which the Parrondo effect appears when $\gamma=1/2$ or $r=s=2$, thereby providing a new example of \textit{Parrondo's paradox} (Harmer and Abbott \cite{HA02}, Abbott \cite{A10}).  

Ethier and Lee \cite{EL12c} studied the random mixture case with $\gamma=1/2$. In this paper we focus on the nonrandom pattern case.  Denoting the mean profits per turn to the ensemble of $N$ players by $\mu_{(\gamma,1-\gamma)}^N$ and $\mu_{[r,s]}^N$ in the two cases of game $C$ and by $\mu_B^N$ in the case of game $B$, it was shown in \cite{EL12c} that $\mu_B^N$ converges under certain conditions on the parameters, and that $\mu_{(1/2,1/2)}^N$ converges essentially always.  The limits can be described in terms of a parameterized spin system on the one-dimensional integer lattice.  Of course one can get similar results for $\mu_{(\gamma,1-\gamma)}^N$ for $0<\gamma<1$.  Here we show that $\mu_{[r,s]}^N$ converges, under certain conditions, to the same limit that $\mu_{(\gamma,1-\gamma)}^N$ converges to, where $\gamma:=r/(r+s)$.  A similar phenomenon is present in a nonspatial $N$-player model of Toral \cite{T02}, as shown in \cite{EL12a}, although in that setting, $\mu_{(\gamma,1-\gamma)}^N$ does not depend on $N$.

Numerical studies \cite{EL12b,EL12d} suggest that $\mu_{[r,s]}^N$ converges much more slowly than $\mu_{(\gamma,1-\gamma)}^N$.  For example, let us consider the special case $p_0=1/10$, $p_1=p_2=3/5$, and $p_3=3/4$.  By $N=18$ (the largest $N$ for which computations have been done in the nonrandom-pattern case), $\mu_{[1,1]}^N$ matches its limiting value to only two significant digits.  On the other hand, by $N=19$ (the largest $N$ for which computations have been done in the random-mixture case), $\mu_B^N$ has stabilized to four significant digits and $\mu_{(1/2,1/2)}^N$ has stabilized to 11 significant digits.

As in \cite{EL12c}, we consider separately a particular choice of the probability parameters, namely $p_0=1$, $p_1=p_2\in(1/2,1)$, and $p_3=0$.  We show that the Parrondo effect (i.e., $\mu_B^N\le0$ and $\mu_{[r,s]}^N>0$) is present if and only if $N$ is even, at least when $s=1$.

Section~\ref{Markov} describes the $N$-player model and the associated discrete-time Markov chain.  Section~\ref{SLLN} establishes a strong law of large numbers (SLLN) for the sequence of profits to the ensemble of $N$ players playing the nonrandom pattern $A^rB^s$, giving several formulas for $\mu_{[r,s]}^N$.  Section~\ref{p0=1,p3=0} treats the special case in which we can confirm the Parrondo effect for all even $N\ge4$.  Section~\ref{spin} introduces the related spin system and reviews its basic properties.  Finally, Section~\ref{limit} establishes our main result, the convergence of $\mu_{[r,s]}^N$ as $N\to\infty$ to a limit that can be expressed in terms of the spin system.

\section{The discrete-time Markov chain}\label{Markov}

Let us define the Markov chain, introduced by Mihailovi\'c and Rajkovi\'c \cite{MR03}, that keeps track of the status (loser or winner, 0 or 1) of each of the $N$ players playing game $B$.  It depends on an integer parameter $N\ge3$ and four probability parameters $p_0,p_1,p_2,p_3\in[0,1]$.  Its state space is the product space 
$$
\Sigma:=\{\bm x=(x_1,x_2,\ldots,x_N): x_i\in\{0,1\}{\rm\ for\ }i=1,\ldots,N\}=\{0,1\}^N
$$
with $2^N$ states.  Let $m_i(\bm x):=2x_{i-1}+x_{i+1}$, or, in other words, $m_i(\bm x)$ is the integer (0, 1, 2, or 3) whose binary representation is $(x_{i-1}\,x_{i+1})_2$;  of course, $x_0:=x_N$ and $x_{N+1}:=x_1$.  Also, let $\bm x^i$ be the element of $\Sigma$ equal to $\bm x$ except at the $i$th component; for example, $\bm x^1:=(1-x_1,x_2,x_3,\ldots,x_N)$.  The one-step transition matrix $\bm P_B$ for this Markov chain has the form
\begin{equation*}
P_B(\bm x,\bm x^i):=\begin{cases}N^{-1}p_{m_i(\bm x)}&\text{if $x_i=0$,}\\N^{-1}q_{m_i(\bm x)}&\text{if $x_i=1$,}\end{cases}\qquad i=1,\ldots,N,\;\bm x\in\Sigma,
\end{equation*}
\begin{equation*}
P_B(\bm x,\bm x):=N^{-1}\bigg(\sum_{i:x_i=0}q_{m_i(\bm x)}+\sum_{i:x_i=1}p_{m_i(\bm x)}\bigg),\qquad \bm x\in\Sigma,
\end{equation*}
where $q_m:=1-p_m$ for $m=0,1,2,3$ and empty sums are 0, and $P_B(\bm x,\bm y)=0$ otherwise. 

Necessary and sufficient conditions on $N\ge3$ and $p_0,p_1,p_2,p_3\in[0,1]$ for the ergodicity of the Markov chain were given in \cite{EL12c}.  (A Markov chain is \textit{ergodic} if there is a unique stationary distribution and the distribution at time $n$ converges to it as $n\to\infty$, regardless of the initial distribution.)

If $p_0=p_1=p_2=p_3=1/2$, then we denote $\bm P_B$ by $\bm P_A$.  Our main concern is with the nonrandom pattern of games $A^rB^s$ for positive integers $r$ and $s$, in which case the relevant Markov chain in $\Sigma$ has one-step transition matrix $\bm P_A^r\bm P_B^s$.  Our first result shows that this chain is ergodic for all choices of the parameters.  Actually, we need a slightly stronger result.

\begin{lemma}\label{ergodic}
Let $N\ge3$ and $p_0,p_1,p_2,p_3\in[0,1]$. Fix $r,s\ge1$ and put $\bm P_1:={\bm P}_A^{r-1}{\bm P}_B^s\bm P_A$, \dots, $\bm P_r:={\bm P}_B^s{\bm P}_A^r$, $\bm P_{r+1}:={\bm P}_B^{s-1}{\bm P}_A^r\bm P_B$, \dots, and $\bm P_{r+s}:={\bm P}_A^r{\bm P}_B^s$.  (These are the $r+s$ cyclic permutations of ${\bm P}_A^r{\bm P}_B^s$.)

\emph{($i$)} The Markov chain in $\Sigma$ with one-step transition matrix $\bm P_1$, $\bm P_2$, \dots, or $\bm P_r$ is irreducible and aperiodic.  In particular, it is ergodic. 

\emph{($ii$)} The Markov chain in $\Sigma$ with one-step transition matrix $\bm P_{r+1}$, $\bm P_{r+2}$, \dots, or $\bm P_{r+s}$ has the following behavior.  There exists a (possibly empty) proper subset $T\subset\Sigma$ such that $T$ is transient and $\Sigma-T$ is closed, irreducible, and aperiodic.  In particular, the Markov chain is ergodic.  In fact, the set $T$, which does not depend on $r$ or $s$, can be specified as follows.

\emph{(a)} If $p_0,p_3\in(0,1)$, then $T=\varnothing$.

\emph{(b)} If $p_0=1$ and $p_3\in(0,1]$, then $T=\{\bm0\}$, with one exception.  If $N$ is divisible by $3$ and $(p_0,p_1,p_2,p_3)=(1,0,0,1)$, then $T=\{\bm0,011\cdots011,101\cdots101,\linebreak110\cdots110\}$.

\emph{(c)} If $p_0=0$ and $p_3\in(0,1)$, then $T=\varnothing$, with one exception.  If $N$ is divisible by $3$ and $p_1=p_2=1$, then $T=\{001\cdots001,010\cdots010,100\cdots100\}$.

\emph{(d)} If $p_0\in[0,1)$ and $p_3=0$, then $T=\{\bm1\}$, with one exception.  If $N$ is divisible by $3$ and $(p_0,p_1,p_2,p_3)=(0,1,1,0)$, then $T=\{001\cdots001,010\cdots010,\linebreak100\cdots100,\bm1\}$.

\emph{(e)} If $p_0\in(0,1)$ and $p_3=1$, then $T=\varnothing$, with one exception.  If $N$ is divisible by $3$ and $p_1=p_2=0$, then $T=\{011\cdots011,101\cdots101,110\cdots110\}$.

\emph{(f)} If $p_0=1$ and $p_3=0$, then $T=\{\bm0,\bm1\}$.

\emph{(g)} Let $p_0=0$ and $p_3=1$.  If $N$ is odd, then $T=\varnothing$, and if $N$ is even, then $T=\{01\cdots01,10\cdots10\}$, with two exceptions.  If $p_1=p_2=0$, then $T$ comprises all states in which $0$s occur as singletons and $1$s occur as singletons or pairs.  If $p_1=p_2=1$, then $T$ comprises all states in which $1$s occur as singletons and $0$s occur as singletons or pairs.
\end{lemma}

\begin{proof}
We claim that it is enough for part ($ii$) to show that
\begin{eqnarray}\label{aper}
P_B(\bm x,\bm x)&>&0,\qquad \bm x\in\Sigma-T,\\ \label{closed}
P_B(\bm x,\bm y)&=&0,\qquad \bm x\in\Sigma,\;\bm y\in T.
\end{eqnarray}
We treat the case of $\bm P_{r+s}$, the cases of $\bm P_{r+1},\ldots,\bm P_{r+s-1}$ being similar.
To see this, suppose $\bm x,\bm y\in\Sigma-T$.  Then there exist $\bm x=\bm x_0,\bm x_1,\ldots,\bm x_{n-1},\bm x_n=\bm y$ such that $\bm x_i\in\Sigma$ and $P_A(\bm x_{i-1},\bm x_i)>0$ for $i=1,2,\ldots,n$ because $\bm P_A$ is irreducible.  We claim that $\bm x_0,\bm x_1,\ldots,\bm x_n$ can be assumed to belong to $\Sigma-T$ (see the paragraph following the next one).  But then
\begin{equation*}
P_{r+s}(\bm x_{i-1},\bm x_i)\ge P_A(\bm x_{i-1},\bm x_i)(1/2)^{r-1}[P_B(\bm x_i,\bm x_i)]^s>0
\end{equation*}
by (\ref{aper}) (and since $P_A(\bm x_i,\bm x_i)=1/2$), so $\Sigma-T$ is irreducible for $\bm P_{r+s}$.  Similarly, if $\bm x\in\Sigma-T$,
\begin{equation*}
P_{r+s}(\bm x,\bm x)\ge(1/2)^r[P_B(\bm x,\bm x)]^s>0
\end{equation*}
by (\ref{aper}), so $\Sigma-T$ is also aperiodic for $\bm P_{r+s}$.  By (\ref{closed}) and the fact that the final factor in $\bm P_{r+s}$ is $\bm P_B$, $\Sigma-T$ is closed and states in $T$ are transient.

We claim that (\ref{aper}) and (\ref{closed}) are also sufficient for part ($i$).  We treat the case of $\bm P_r$, the cases of $\bm P_1,\ldots,\bm P_{r-1}$ being similar.  To see this, suppose $\bm x,\bm y\in\Sigma$.  By (\ref{closed}), there exists $\bm x'\in\Sigma-T$ such that $P_B(\bm x,\bm x')>0$.  Hence
\begin{equation*}
P_r(\bm x,\bm x')\ge P_B(\bm x,\bm x')[P_B(\bm x',\bm x')]^{s-1}(1/2)^r>0.
\end{equation*}
Also, because of the simple form that $T$ has, there exists $\bm y'\in\Sigma-T$ such that $P_A(\bm y',\bm y)>0$.  Finally, as in the preceding paragraph, there exist $\bm x'=\bm x_0,\bm x_1,\ldots,\bm x_{n-1}=\bm y',\bm x_n=\bm y$ such that $\bm x_i\in\Sigma-T$ and $P_A(\bm x_{i-1},\bm x_i)>0$ for $i=1,2,\ldots,n-1$.  We then get
\begin{equation*}
P_r(\bm x_{i-1},\bm x_i)\ge[P_B(\bm x_{i-1},\bm x_{i-1})]^s(1/2)^{r-1}P_A(\bm x_{i-1},\bm x_i)>0
\end{equation*}
for $i=1,2,\ldots,n$, so $\bm P_r$ is irreducible.  Finally, for the aperiodicity of $\bm P_r$, we observe that, if $\bm x\in\Sigma-T$,
\begin{equation*}
P_r(\bm x,\bm x)\ge[P_B(\bm x,\bm x)]^s(1/2)^r>0,
\end{equation*}
and this suffices by irreducibility.

There is a missing step in the first paragraph.  Specifically, we must show that, given $\bm x,\bm y\in\Sigma-T$, there exist $\bm x=\bm x_0,\bm x_1,\ldots,\bm x_{n-1},\bm x_n=\bm y$ such that $\bm x_i\in\Sigma-T$ and $P_A(\bm x_{i-1},\bm x_i)>0$ for $i=1,2,\ldots,n$.  In fact, we can choose $n$ equal to the Hamming distance between $\bm x$ and $\bm y$, $d(\bm x,\bm y):=\sum_{i=1}^N|x_i-y_i|$.  The justification for this requires a case-by-case analysis, but the idea is much the same in each case.  Suppose $T=\{\bm0\}$.  Then $\bm y$ must have $y_k=1$ for some $k\in\{1,2,\ldots,N\}$.  If $x_k=1$, then every $\bm x_i$ will have $k$th component 1, hence will not be in $T$.  If $x_k=0$, then let $\bm x_1=\bm x^k$.  Again, $\bm x_i$ will have $k$th component equal to 1 for $i=1,\ldots,n$, hence will not be in $T$ along with $\bm x_0$ by assumption.  

A similar argument works for $T=\{\bm1\}$, so suppose $T=\{\bm0,\bm1\}$.  Let $k\in\{1,2,\ldots,N\}$ be such that $(y_k,y_{k+1})=(0,1)$.  If $(x_k,x_{k+1})$ equals $(0,1)$, we are finished; if it equals $(0,0)$ or $(1,1)$, then $\bm x_1$ is $\bm x^k$ or $\bm x^{k+1}$ as needed.  Finally, if $(x_k,x_{k+1})=(1,0)$, then define $\bm x_1$ and $\bm x_2$ by flipping the bits at sites $k$ and $k+1$ in whichever order is necessary to avoid having $\bm x_1\in T$.  It follows that $\bm x_0,\bm x_1,\ldots,\bm x_n\in\Sigma-T$. 

Next, suppose $N$ is even and $T=\{01\cdots01,10\cdots10\}$.  Given $\bm x,\bm y\in\Sigma-T$, $\bm y$ must have two consecutive 0s or two consecutive 1s; assume the former case, the latter case being symmetric.  Then there exists $k\in\{1,2,\ldots,N\}$ such that $(y_k,y_{k+1})=(0,0)$.  The argument is now completed as in the preceding paragraph.

There are six other cases that must be considered (the various exceptional cases in the lemma).  By symmetry, it is enough to consider the case $T=\{001\cdots001,010\cdots010,100\cdots100\}$, the case where $T$ is the union of the latter set and $\bm1$, and the case where $T$ contains all vectors in which 1s occur only as singletons and 0s occur only as singletons or pairs.  We treat the first case, the other two being similar.  Given $\bm x,\bm y\in\Sigma-T$, $\bm y$ must have a segment of the form 000, 011, 101, 110, or 111.  For example, in the first case, there is a $k\in\{1,2,\ldots,N\}$ such that $(y_k,y_{k+1},y_{k+2})=(0,0,0)$.  In any case, if $\bm x$ differs from $\bm y$ at none of these three sites, we are finished.  If it differs at one, we flip the bit at that site to determine $\bm x_1$.  If it differs at two, we flip the bits at these two sites to determine $\bm x_1$ and $\bm x_2$, the order chosen so as to avoid having $\bm x_1\in T$.  If it differs at all three sites, we flip the bit at one of the three sites to determine $\bm x_1$, the site chosen to avoid $\bm x_1\in T$.  Then it differs at two of the sites, a case we have already treated.  This completes the missing step.

It remains to show that (\ref{aper}) and (\ref{closed}) are satisfied by $\bm P_B$.  We consider (\ref{aper}) first.  By virtue of
\begin{equation*}
P_B(\bm x,\bm x):=N^{-1}\bigg(\sum_{i:x_i=0}q_{m_i(\bm x)}+\sum_{i:x_i=1}p_{m_i(\bm x)}\bigg),
\end{equation*}
property (\ref{aper}) holds if, for each $\bm x\in\Sigma-T$, at least one of the following holds:  $p_0<1$ and $\bm x$ contains the segment $000$; $p_0>0$ and $\bm x$ contains the segment $010$; $p_1<1$ and $\bm x$ contains the segment $001$; $p_1>0$ and $\bm x$ contains the segment $011$; $p_2<1$ and $\bm x$ contains the segment $100$; $p_2>0$ and $\bm x$ contains the segment $110$; $p_3<1$ and $\bm x$ contains the segment $101$; $p_3>0$ and $\bm x$ contains the segment $111$.  

(a)  If $\bm x$ has three consecutive $0$s or three consecutive $1$s, then $p_0<1$ or $p_3>0$ suffice, so suppose not.  Then $\bm x$ contains the pair $01$.  If $01$ is part of $010$ or $101$, then $p_0>0$ or $p_3<1$ suffice, so we can assume that $01$ is part of $0011$, hence $100110$.  This suffices if $p_2<1$, $p_1<1$, $p_1>0$, or $p_2>0$, and at least two of these inequalities must hold.  

(b)  If $p_3<1$, then the argument is similar to that of (a), except $\bm x=\bm0$ is excluded but we cannot rule out three consecutive $0$s.  So $01$ is part of $0011$, hence part of $100\cdots00110$.  The proof is otherwise unchanged.

If $p_3=1$, then $\bm x=\bm0$ is excluded but we cannot rule out three consecutive $0$s.  If $\bm x$ has three consecutive $1$s, then $p_3>0$ suffices, so suppose not.  Then $\bm x$ contains the pair $01$.  If $01$ is part of $010$, then $p_0>0$ suffices.  Hence we can assume it is part of $011$, therefore $0110$.  We are finished if $p_1>0$ or $p_2>0$, so suppose $p_1=p_2=0$.  If we exclude $\bm x$ of the form $011\cdots011$, $101\cdots101$, or $110\cdots110$, then we have ruled out all possibilities (singleton $1$s and three or more consecutive $1$s are excluded, and two or more consecutive $0$s can be excluded because $p_1<1$ and $p_2<1$).

(c)  If $\bm x$ has three consecutive $0$s or three consecutive $1$s, then $p_0<1$ or $p_3>0$ suffice, so suppose not.  Then $\bm x$ contains the pair $01$.  If $01$ is part of $101$, then $p_3<1$ suffices, so suppose not.  Then it is part of $001$, therefore $1001$.  We are finished if $p_2<1$ or $p_1<1$, so suppose $p_1=p_2=1$.  If we exclude $\bm x$ of the form $001\cdots001$, $010\cdots010$, or $100\cdots100$, we have ruled out all possibilities (singleton $0$s and three or more consecutive $0$s are excluded, and two or more consecutive $1$s can be excluded because $p_1>0$ and $p_2>0$).

(d) and (e)  These cases are symmetric with (b) and (c).

(f)  The argument is similar to that of (a), except $\bm x=\bm0$ and $\bm x=\bm1$ are excluded but we cannot rule out three consecutive $0$s or three consecutive $1$s.  Then $\bm x$ contains the pair $01$.  If $01$ is part of $010$ or $101$, then $p_0>0$ or $p_3<1$ suffice, so we can assume that $01$ is part of $0011$, hence $100\cdots0011\cdots10$.  This is similar to case (a).  

(g)  If $\bm x$ has three consecutive $0$s or three consecutive $1$s, then $p_0<1$ or $p_3>0$ suffice.  So suppose not.  Then $\bm x$ contains the pair $01$.  Suppose it is part of $001$ or $011$.  In the first case it is part of $1001$.  If this is part of $10011$, then it suffices that $p_2<1$, $p_1<1$, or $p_1>0$, at least one of which must hold.  If this is part of $11001$, then it suffices that $p_2>0$, $p_2<1$, or $p_1<1$, at least one of which must hold.  Therefore we can assume that $1001$ is part of $010010$.  This suffices if $p_2<1$ or $p_1<1$, but if $p_1=p_2=1$, then we must rule out states in which $1$s occur as singletons and $0$s occur as singletons or pairs.  If $\bm x$ is not of this form, then we can choose our initial $01$ in such a way that it is not embedded in $010010$.  In the second case, in which $01$ is part of $011$, it is part of $0110$.  If this is part of $01100$, then it suffices that $p_1>0$, $p_2>0$, or $p_2<1$, at least one of which must hold.  If this is part of $00110$, then it suffices that $p_1<1$, $p_1>0$, or $p_2>0$, at least one of which must hold.  Therefore we can assume that $0110$ is part of $101101$.  This suffices if $p_1>0$ or $p_2>0$, but if $p_1=p_2=0$, then we must rule out states in which $0$s occur as singletons and $1$s occur as singletons or pairs.  If $\bm x$ is not of this form, then we can choose our initial $01$ in such a way that it is not embedded in $101101$.  Finally, the only other possibility is that $01$ is part of $1010$.  Assuming $\bm x$ is not part of $01\cdots01$ or $10\cdots10$ with $N$ even, there must be a $01$ that is not embedded in $1010$.

This finally proves (\ref{aper}), so we turn to (\ref{closed}), which is equivalent to
\begin{equation*}
P_B(\bm y,\bm y)=0,\quad P_B(\bm y^i,\bm y)=0,\qquad \bm y\in T,\; i=1,2,\ldots,N.
\end{equation*}
If $\bm0\in T$, then $p_0=1$ and it suffices to note that $P_B(\bm0,\bm0)=0$ and $P_B(\bm0^i,\bm0)=0$.  If $\bm1\in T$, then $p_3=0$ and it suffices to note that $P_B(\bm1,\bm1)=0$ and $P_B(\bm1^i,\bm1)=0$.  If $N$ is even and $01\cdots01\in T$, then $p_0=0$ and $p_3=1$ and it suffice to note that $P_B(01\cdots01,01\cdots01)=0$ and $P_B(01\cdots01^i,01\cdots01)=0$.  The same applies if $N$ is even and $10\cdots10\in T$.  If $N$ is divisible by 3 and  $001\cdots001\in T$, then $p_0=0$ and $p_1=p_2=1$ and it suffices to note that $P_B(001\cdots001,001\cdots001)=0$ and $P_B(001\cdots001^i,001\cdots001)=0$.  This also applies to rotations of $001\cdots001$.  If $N$ is divisible by 3 and  $011\cdots011\in T$, then $p_1=p_2=0$ and $p_3=1$ and it suffices to note that $P_B(011\cdots011,011\cdots011)=0$ and $P_B(011\cdots011^i,011\cdots011)=0$.  This also applies to rotations of $011\cdots011$.  The only remaining cases are the exceptional cases of part (g).  If $(p_0,p_1,p_2,p_3)=(0,0,0,1)$ and if $\bm x$ has only singleton 0s and singleton or paired 1s, then $P_B(\bm x,\bm x)=0$ and $P_B(\bm x^i,\bm x)=0$.  If $(p_0,p_1,p_2,p_3)=(0,1,1,1)$ and if $\bm x$ has only singleton or paired 0s and singleton 1s, then $P_B(\bm x,\bm x)=0$ and $P_B(\bm x^i,\bm x)=0$.
\end{proof}

\begin{lemma}\label{invariance}
Let $G$ be a subgroup of the symmetric group $S_N$.  Let $\bm P$ be the one-step transition matrix for a Markov chain in $\Sigma$ having a unique stationary distribution $\bm\pi$.  For $\bm x=(x_1,\ldots,x_N)\in\Sigma$ and $\sigma\in G$, write $\bm x_\sigma:=(x_{\sigma(1)},\ldots,x_{\sigma(N)})$, and assume that
\begin{equation}\label{G-invariance}
P(\bm x_\sigma,\bm y_\sigma)=P(\bm x,\bm y),\qquad \sigma\in G,\; \bm x,\bm y\in\Sigma. 
\end{equation}
Then $\pi(\bm x_\sigma)=\pi(\bm x)$ for all $\sigma\in G$ and $\bm x\in\Sigma$.  
\end{lemma}

This lemma is from \cite{EL12b}, where it was shown to apply to $\bm P_B$ when $G$ is the subgroup of cyclic permutations (or rotations) of $(1,2,\ldots,N)$ and, if $p_1=p_2$, when $G$ is the subgroup generated by the cyclic permutations and the order-reversing permutation (rotations and/or reflections) of $(1,2,\ldots,N)$, the dihedral group of order $2N$.  It therefore also applies to $\bm P_A$ under the same conditions and hence, for fixed $r,s\ge1$, to each of the one-step transition matrices $\bm P_1,\bm P_2,\ldots,\bm P_{r+s}$ of Lemma \ref{ergodic} under the same conditions.  For this we need a simple observation.  Let us define the stochastic matrix $\bm P$ with rows and columns indexed by $\Sigma$ to be \textit{$G$-invariant} if (\ref{G-invariance}) holds.  We notice that the class of $G$-invariant stochastic matrices is closed under matrix multiplication, for if $\bm P_1$ and $\bm P_2$ are $G$-invariant, then
\begin{eqnarray*}
[\bm P_1\bm P_2](\bm x_\sigma,\bm y_\sigma)&=&\sum_{\bm z\in\Sigma}P_1(\bm x_\sigma,\bm z)P_2(\bm z,\bm y_\sigma)=\sum_{\bm z\in\Sigma}P_1(\bm x_\sigma,\bm z_\sigma)P_2(\bm z_\sigma,\bm y_\sigma)\\
&=&\sum_{\bm z\in \Sigma}P_1(\bm x,\bm z)P_2(\bm z,\bm y)=[\bm P_1 \bm P_2](\bm x,\bm y).
\end{eqnarray*}
for all $\sigma\in G$ and $\bm x,\bm y\in\Sigma$.

For example, with $\bm\pi$ being the unique stationary distribution of $\bm P_A^r\bm P_B^s$, the lemma applies to $\bm\pi\bm P_A^r\bm P_B^v$, which is the unique stationary distribution of $\bm P_{r+v}$, for $v=0,1,\ldots,s-1$.  We conclude that, if $p_1=p_2$, then the $1,3$ two-dimensional marginals of $\bm\pi\bm P_A^r\bm P_B^v$ satisfy
\begin{equation}\label{1,3 symm}
[\bm\pi\bm P_A^r\bm P_B^v]_{1,3}(0,1)=[\bm\pi\bm P_A^r\bm P_B^v]_{1,3}(1,0),\qquad v=0,1,\ldots,s-1.
\end{equation}

\section{SLLN}\label{SLLN}

We will need the following version of the strong law of large numbers from \cite{EL09}.

\begin{theorem}\label{SLLN-EL09}
Let $\bm P_A$ and $\bm P_B$ be one-step transition matrices for Markov chains in a finite state space $\Sigma_0$.  Fix $r,s\ge1$.  Assume that $\bm P:=\bm P_A^r\bm P_B^s$, as well as all cyclic permutations of $\bm P_A^r\bm P_B^s$, are ergodic, and let the row vector $\bm\pi$ be the unique stationary distribution of $\bm P$.  Given a real-valued function $w$ on $\Sigma_0\times\Sigma_0$, define the payoff matrix $\bm W:=(w(i,j))_{i,j\in\Sigma_0}$.  Define $\dot{\bm P}_A:=\bm P_A\circ\bm W$ and $\dot{\bm P}_B:=\bm P_B\circ\bm W$, where $\circ$ denotes the Hadamard (entrywise) product, and put
$$
\mu_{[r,s]}:={1\over r+s}\bigg[\sum_{u=0}^{r-1}\bm\pi\bm P_A^u\dot{\bm P}_A\bm1+\sum_{v=0}^{s-1}\bm\pi\bm P_A^r\bm P_B^v\dot{\bm P}_B\bm1\bigg],
$$
where $\bm1$ denotes a column vector of $1$s with entries indexed by $\Sigma_0$.  Let $\{X_n\}_{n\ge0}$ be a nonhomogeneous Markov chain in $\Sigma_0$ with one-step transition matrices $\bm P_A,\ldots,\bm P_A$ $(r\text{ times})$, $\bm P_B,\ldots,\bm P_B$ $(s\text{ times})$, $\bm P_A,\ldots,\bm P_A$ $(r\text{ times})$, $\bm P_B,\ldots,\bm P_B$ $(s\text{ times})$, and so on, and let the initial distribution be arbitrary.  For each $n\ge1$, define $\xi_n:=w(X_{n-1},X_n)$ and $S_n:=\xi_1+\cdots+\xi_n$.  Then $\lim_{n\to\infty}n^{-1}S_n=\mu_{[r,s]}$ {\rm a.s.}
\end{theorem}

\begin{remark}
Under an additional assumption there is also a central limit theorem.
\end{remark}

Theorem \ref{SLLN-EL09} applies not to $\bm P_A$ and $\bm P_B$ of Section~\ref{Markov} but to analogous one-step transition matrices on a slightly more informative state space.  The new state space is
$\Sigma^*:=\Sigma\times\{1,2,\ldots,N\}$ and the process is in state $(\bm x,i)$ if $\bm x$ describes the status of each player and $i$ is the next player to play.  Given $N\ge3$ and $p_0,p_1,p_2,p_3\in[0,1]$, we define $\bm P_B^*$ by
\begin{equation*}
P_B^*((\bm x,i),(\bm x^i,j)):=\begin{cases}N^{-1}p_{m_i(\bm x)}&\text{if $x_i=0$,}\\N^{-1}q_{m_i(\bm x)}&\text{if $x_i=1$,}\end{cases}
\end{equation*}
\begin{equation*}
P_B^*((\bm x,i),(\bm x,j)):=\begin{cases}N^{-1}q_{m_i(\bm x)}&\text{if $x_i=0$,}\\N^{-1}p_{m_i(\bm x)}&\text{if $x_i=1$,}\end{cases}
\end{equation*}
for all $(\bm x,i)\in\Sigma^*$ and $j=1,2,\ldots,N$, where $q_m:=1-p_m$ for $m=0,1,2,3$, and $P_B^*((\bm x,i),(\bm y,j))=0$ otherwise.

We further define $\bm P_A^*$ to be $\bm P_B^*$ with $p_0=p_1=p_2=p_3=1/2$.

\begin{lemma}
Let $N\ge3$ and $p_0,p_1,p_2,p_3\in[0,1]$. Fix $r,s\ge1$ and put $\bm P_1^*:=({\bm P}_A^*)^{r-1}({\bm P}_B^*)^s\bm P_A^*$, \dots, $\bm P_r^*:=({\bm P}_B^*)^s({\bm P}_A^*)^r$, $\bm P_{r+1}^*:=({\bm P}_B^*)^{s-1}({\bm P}_A^*)^r\bm P_B^*$, \dots, and $\bm P_{r+s}^*:=({\bm P}_A^*)^r({\bm P}_B^*)^s$.  (These are the $r+s$ cyclic permutations of $({\bm P}_A^*)^r({\bm P}_B^*)^s$.)

\emph{($i$)} The Markov chain in $\Sigma^*$ with one-step transition matrix $\bm P_1^*$, $\bm P_2^*$, \dots, or $\bm P_r^*$ is irreducible and aperiodic.  In particular, it is ergodic. 

\emph{($ii$)} The Markov chain in $\Sigma^*$ with one-step transition matrix $\bm P_{r+1}^*$, $\bm P_{r+2}^*$, \dots, or $\bm P_{r+s}^*$ has the following behavior.  There exists a (possibly empty) proper subset $T\subset\Sigma$ such that $T\times\{1,2,\ldots,N\}$ is transient and $(\Sigma-T)\times\{1,2,\ldots,N\}$ is closed, irreducible, and aperiodic.  In particular, the Markov chain is ergodic.  In fact, the set $T$, which does not depend on $r$ or $s$, is as in Lemma \ref{ergodic}.  

Let $\bm\pi^*$ denote the unique stationary distribution for $({\bm P}_A^*)^r({\bm P}_B^*)^s$, and let $\bm\pi$ denote the unique stationary distribution for $\bm P_A^r\bm P_B^s$.  Then 
\begin{equation}\label{indep}
\bm\pi^*(\bm P_A^*)^r(\bm P_B^*)^v=\bm\pi\bm P_A^r\bm P_B^v\times{\rm uniform}\{1,2,\ldots,N\}
\end{equation}
for $v=0,1,\ldots,s-1$.  Also, $\bm\pi^*=\bm\pi\times{\rm uniform}\{1,2,\ldots,N\}$.
\end{lemma}

\begin{proof}
Let $\bm\pi^*$ be stationary for $(\bm P_A^*)^r(\bm P_B^*)^s$.  We will show that it has the form stated in the lemma.  Let $\bm X^*(0),\bm X^*(1),\ldots$ be a nonhomogeneous Markov chain in $\Sigma^*$ with transition matrices $\bm P_A^*,\ldots,\bm P_A^*$ ($r$ times), $\bm P_B^*,\ldots,\bm P_B^*$ ($s$ times), $\bm P_A^*,\ldots,\bm P_A^*$ ($r$ times), $\bm P_B^*,\ldots,\bm P_B^*$ ($s$ times), and so on, and initial distribution $\bm\pi^*$.  Then $\bm X^*(r+s)$ has distribution $\bm\pi^*(\bm P_A^*)^r(\bm P_B^*)^s=\bm\pi^*$.  On the other hand, writing $\bm X^*(n)=:(\bm X(n),I(n))$ for each $n\ge0$, we claim that, for each $n\ge1$, $I(n)$ is independent of $\bm X(n)$ and is uniform$\{1,2,\ldots,N\}$, conditionally on $(\bm X(n-1),I(n-1))$, hence also unconditionally.  This follows from $P_B^*((\bm x,i),(\bm x^i,j))=N^{-1}c_i(\bm x)$ and $P_B^*((\bm x,i),(\bm x,j))=N^{-1}[1-c_i(\bm x)]$, where $c_i(\bm x):=p_{m_i(\bm x)}$ if $x_i=0$ and $c_i(\bm x):=q_{m_i(\bm x)}$ if $x_i=1$ (and similarly for $\bm P_A^*$ but with $p_0=p_1=p_2=p_3=1/2$).  Since $\bm X^*(0)$ has the same distribution as $\bm X^*(r+s)$, we find that $I(n)$ is independent of $\bm X(n)$ and is uniform$\{1,2,\ldots,N\}$ for each $n\ge0$ (not just $n\ge1$).  We next claim that $\bm X(0),\bm X(1),\ldots$ is a nonhomogeneous Markov chain in $\Sigma$ with transition matrices $\bm P_A,\ldots,\bm P_A$ ($r$ times), $\bm P_B,\ldots,\bm P_B$ ($s$ times), $\bm P_A,\ldots,\bm P_A$ ($r$ times), $\bm P_B,\ldots,\bm P_B$ ($s$ times), and so on, and initial distribution $\bm\pi$, $\bm x$-marginal of $\bm\pi^*$.  The Markov property is essentially a consequence of identities such as
\begin{eqnarray*}
&&\P(\bm X(r+s)=\bm x^i\mid\bm X(r+s-1)=\bm x)\\
&&\quad{}={\P(\bm X(r+s)=\bm x^i,\bm X(r+s-1)=\bm x)\over\P(\bm X(r+s-1)=\bm x)}\\
&&\quad{}={\P((\bm X(r+s),I(r+s))=(\bm x^i,\cdot),(\bm X(r+s-1),I(r+s-1))=(\bm x,i))\over\P((\bm X(r+s-1),I(r+s-1))=(\bm x,\cdot))}\\
&&\quad{}={\P((\bm X(r+s),I(r+s))=(\bm x^i,j),(\bm X(r+s-1),I(r+s-1))=(\bm x,i))\over\P((\bm X(r+s-1),I(r+s-1))=(\bm x,i))}\\
&&\quad{}=\P((\bm X(r+s),I(r+s))=(\bm x^i,j)\mid(\bm X(r+s-1),I(r+s-1))=(\bm x,i))\\
&&\quad{}=P_B^*((\bm x,i),(\bm x^i,j))\\
&&\quad{}=P_B(\bm x,\bm x^i),
\end{eqnarray*}
where, for example, $I(r+s)=\cdot$ means that the value of $I(r+s)$ is unspecified.

Since $\bm X(r+s)$ has distribution $\bm\pi\bm P_A^r\bm P_B^s$ as well as distribution $\bm\pi$, we see that $\bm\pi$ is the unique stationary distribution for $\bm P_A^r\bm P_B^s$, as assumed in the statement of the lemma.  Finally, $\bm\pi^*$, being the distribution of $\bm X^*(0)=(\bm X(0),I(0))$, must equal $\bm\pi\times{\rm uniform}\{1,2,\ldots,N\}$, and the last conclusion of the lemma follows.  For $v=0,1,\ldots,s-1$, $(\bm X(r+v),I(r+v))$ has distribution $\bm\pi^*(\bm P_A^*)^r(\bm P_B^*)^v$ while $\bm X(r+v)$ has distribution $\bm\pi\bm P_A^r\bm P_B^v$, so by the independence result of the preceding paragraph, (\ref{indep}) follows.

It remains to prove the assertions about $\bm P_1^*,\ldots,\bm P_{r+s}^*$.  Let us first treat the case of $\bm P_{r+s}^*$, the cases of $\bm P_{r+1}^*,\ldots,\bm P_{r+s-1}^*$ being similar.  If $i\in\{1,2,\ldots,N\}$ and $(\bm y,j)\in T\times\{1,2,\ldots,N\}$, then $P_B^*((\bm y^i,i),(\bm y,j))=P_B(\bm y^i,\bm y)=0$ and $P_B^*((\bm y,i),(\bm y,j))\le P_B(\bm y,\bm y)=0$, so for all $(\bm x,i)\in\Sigma^*$ and $(\bm y,j)\in T\times\{1,2,\ldots,N\}$, $P_{r+s}^*((\bm x,i),(\bm y,j))=0$.  This implies the transience of $T\times\linebreak\{1,2,\ldots,N\}$ and the closedness of $(\Sigma-T)\times\{1,2,\ldots,N\}$.  As for the irreducibility of $(\Sigma-T)\times\{1,2,\ldots,N\}$, let $(\bm x,i)$ and $(\bm y,j)$ belong to this set.  Let $\bm x_0:=\bm x$ and let $\bm x_1\in\Sigma-T$ be such that $P_{r+s}^*((\bm x_0,i),(\bm x_1,k))>0$ for all $k\in\{1,2,\ldots,N\}$.  By the irreducibility of $\bm P_{r+s}$ on $\Sigma-T$ (Lemma \ref{ergodic}), there exist $\bm x_2,\ldots,\bm x_n=\bm y$ such that $\bm x_l\in\Sigma-T$, $\bm x_{l-1}\ne\bm x_l$, and $P_{r+s}(\bm x_{l-1},\bm x_l)>0$ for $l=2,\ldots,n$.  Then there also exist $k_1,\ldots,k_{n-1}\in\{1,2,\ldots,N\}$ such that $P_{r+s}^*((\bm x_{l-1},k_{l-1}),(\bm x_l,k_l))>0$ for $l=2,\ldots,n$ with $k_n:=j$.  With $k_0:=i$, this also holds for $l=1$, so we have $P_{r+s}^*((\bm x,i),(\bm y,j))>0$.  For the aperiodicity of $(\Sigma-T)\times\{1,2,\ldots,N\}$, we need only show that, for some $(\bm x,i)$ belonging to this set,
$P_{r+s}^*((\bm x,i),(\bm x,i))>0$.  Let $\bm x\in\Sigma-T$.  Then $P_B(\bm x,\bm x)>0$, so there exists $i_0\in\{1,2,\ldots,N\}$ such that $P_B^*((\bm x,i_0),(\bm x,i_0))>0$.  Hence
\begin{eqnarray*}
P_{r+s}^*((\bm x,i_0),(\bm x,i_0))&\ge&[P_A^*((\bm x,i_0),(\bm x,i_0))]^r[P_B^*((\bm x,i_0),(\bm x,i_0))]^s>0.
\end{eqnarray*}

Finally, we treat the case of $\bm P_r^*$, the cases of $\bm P_1^*,\ldots,\bm P_{r-1}^*$ being similar.  For irreducibility, let $(\bm x,i)$ and $(\bm y,j)$ belong to $\Sigma^*$.  Let $\bm x_0:=\bm x$ and let $\bm x_1$ be such that $P_r^*((\bm x_0,i),(\bm x_1,k))>0$ for all $k\in\{1,2,\ldots,N\}$.  Then, by the irreducibility of $\bm P_r$ (Lemma \ref{ergodic}), there exist $\bm x_2,\bm x_3,\ldots,\bm x_n=\bm y$ such that $\bm x_{l-1}\ne\bm x_l$ and $P_r(\bm x_{l-1},\bm x_l)>0$ for $l=2,\ldots,n$.  Then there also exist $k_1,\ldots,k_{n-1}\in\{1,2,\ldots,N\}$ such that $P_r^*((\bm x_{l-1},k_{l-1}),(\bm x_l,k_l))>0$ for $l=2,\ldots,n$ with $k_n:=j$.  With $k_0:=i$, this also holds for $l=1$, so we have $P_r^*((\bm x,i),(\bm y,j))>0$.  For aperiodicity, we need only show that, for some $(\bm x,i)\in\Sigma^*$,
$P_r^*((\bm x,i),(\bm x,i))>0$.  Let $\bm x\in\Sigma-T$.  Then the argument is as in the preceding paragraph.
\end{proof}

Notice also that the profit corresponding to each nonzero entry of $\bm P_B^*$ is equal to $\pm1$, so Theorem \ref{SLLN-EL09} applies and there are several formulas for the mean profit, as we now show.

\begin{theorem}\label{SLLN-thm}
Given $r,s\ge1$,
let $\bm\pi$ be the unique stationary distribution for the one-step transition matrix $\bm P_A^r\bm P_B^s$.  Let $\{(\bm X(n), I(n))\}_{n\ge0}$ be a nonhomogeneous Markov chain in $\Sigma^*$ with one-step transition matrices $\bm P_A^*,\ldots,\bm P_A^*$ $(r\text{ times})$, $\bm P_B^*,\ldots,\bm P_B^*$ $(s\text{ times})$, $\bm P_A^*,\ldots,\bm P_A^*$ $(r\text{ times})$, $\bm P_B^*,\ldots,\bm P_B^*$ $(s\text{ times})$, and so on, and arbitrary initial distribution.  Define 
$$
\xi_n:=w((\bm X(n-1),I(n-1)),(\bm X(n),I(n))), \qquad n\ge1,
$$
where the payoff function $w$ is 1 for a win and $-1$ for a loss, determined by whether the corresponding entry of $\bm P_B^*$ is of the form $N^{-1}p_m$ or $N^{-1}q_m$.  Let $S_n:=\xi_1+\cdots+\xi_n$ for each $n\ge1$.  Then $n^{-1}S_n\to\mu_{[r,s]}^N$ {\rm a.s.} as $n\to\infty$, where the mean profit $\mu_{[r,s]}^N$ can be expressed in terms of $\bm\pi\bm P_A^r\bm P_B^v$ as 
\begin{equation}\label{mu1}
\mu_{[r,s]}^N={1\over r+s}\sum_{v=0}^{s-1}\sum_{\bm x\in\Sigma}[\bm\pi\bm P_A^r\bm P_B^v](\bm x){1\over N}\sum_{i=1}^N [p_{m_i(\bm x)}-q_{m_i(\bm x)}],
\end{equation}
in terms of the $1,3$ two-dimensional marginals of $\bm\pi\bm P_A^r\bm P_B^v$ as
\begin{eqnarray}\label{mu2}
\mu_{[r,s]}^N&=&{1\over r+s}\sum_{v=0}^{s-1}\sum_{w=0}^1\sum_{z=0}^1[\bm\pi\bm P_A^r\bm P_B^v]_{1,3}(w,z)(p_{2w+z}-q_{2w+z}),
\end{eqnarray}
or in terms of the one-dimensional marginals of $\bm\pi\bm P_A^u$ and $\bm\pi\bm P_A^r\bm P_B^v$ as
\begin{eqnarray}\label{mu3}
\mu_{[r,s]}^N&=&{1\over r+s}\bigg[\sum_{u=0}^{r-1}\{[\bm\pi\bm P_A^u]_1(1)-[\bm\pi\bm P_A^u]_1(0)\}\\
&&\qquad\qquad{}+\sum_{v=0}^{s-1}\{[\bm\pi\bm P_A^r\bm P_B^v]_1(1)-[\bm\pi\bm P_A^r\bm P_B^v]_1(0)\}\bigg].\nonumber
\end{eqnarray}
In the special case $s=1$, (\ref{mu3}) takes the simpler form
\begin{equation}\label{mu4}
\mu_{[r,1]}^N={N[1-(1-1/N)^{r+1}]\over(r+1)(1-1/N)^r}\{[\bm\pi\bm P_A^r]_1(1)-[\bm\pi\bm P_A^r]_1(0)\}.
\end{equation}
\end{theorem}

\begin{remark}
Another formula for $\mu_{[r,s]}^N$, better suited to numerical computation, was given in \cite{EL12d}.
\end{remark}

\begin{proof}
Theorem~\ref{SLLN-EL09} gives the SLLN with 
$$
\mu_{[r,s]}^N={1\over r+s}\sum_{v=0}^{s-1}\bm\pi^*({\bm P_A^*})^r({\bm P_B^*})^v\dot{\bm P}_B^*\bm1,
$$
where $\dot{\bm P_B^*}$ is $\bm P_B^*$ with each $q_m$ replaced by $-q_m$ and $\bm1$ is a column vector of $1$s indexed by $\Sigma^*$; here we used $\dot{\bm P}_A^*\bm1=\bm0$.  Since $[\dot{\bm P}_B^*\bm1](\bm x,i)=p_{m_i(\bm x)}-q_{m_i(\bm x)}$, this and (\ref{indep}) imply (\ref{mu1}). 

Next, using (\ref{mu1}) and the rotation invariance property (see Lemma~\ref{invariance} and the discussion following it), we have
\begin{eqnarray*}
\mu_{[r,s]}^N&=&{1\over r+s}\sum_{v=0}^{s-1}\sum_{\bm x\in\Sigma}[\bm\pi\bm P_A^r\bm P_B^v](\bm x){1\over N}\sum_{i=1}^N [p_{m_i(\bm x)}-q_{m_i(\bm x)}]\\
&=&{1\over r+s}\sum_{v=0}^{s-1}{1\over N}\sum_{i=1}^N\sum_{w=0}^1\sum_{z=0}^1[\bm\pi\bm P_A^r\bm P_B^v]_{i-1,i+1}(w,z)(p_{2w+z}-q_{2w+z})\\
&=&{1\over r+s}\sum_{v=0}^{s-1}\sum_{w=0}^1\sum_{z=0}^1[\bm\pi\bm P_A^r\bm P_B^v]_{1,3}(w,z)(p_{2w+z}-q_{2w+z}),
\end{eqnarray*}
which is (\ref{mu2}).  In the second line, the $0,2$ and $N-1,N+1$ marginals are the $N,2$ and $N-1,1$ marginals.

Next, turning to (\ref{mu3}), we let $\{(\bm X(n),I(n))\}_{n\in{\bf Z}}$ be a nonhomogeneous Markov chain in $\Sigma^*$ with time parameter ranging over ${\bf Z}$, the set of integers, and with one-step transition matrices $\bm P_A^*$ from $(\bm X(n),I(n))$ if $n$ (mod $r+s$) belongs to $\{0,1,\ldots,r-1\}$ and $\bm P_B^*$ from $(\bm X(n),I(n))$ if $n$ (mod $r+s$) belongs to $\{r,r+1,\ldots,r+s-1\}$.  Assume that $(\bm X(0),I(0))$ has distribution $\bm\pi\times{\rm uniform}\{1,2,\ldots,N\}$.  Then $\{(\bm X((r+s)n+j),I((r+s)n+j))\}_{n\in{\bf Z}}$ is a stationary sequence for each $j\in{\bf Z}$ with $\bm X(j)$ having distribution, for $j=0,1,\ldots,r+s-1$,
$$
\bm\pi^j:=\begin{cases}\bm\pi\bm P_A^j&\text{if $j\in\{0,1,\ldots,r-1\}$,}\\
\bm\pi\bm P_A^r\bm P_B^{j-r}&\text{if $j\in\{r,r+1,\ldots,r+s-1\}$.}\end{cases}
$$
Therefore,
\begin{eqnarray*}
\pi_1^j(1)&=&\pi_2^j(1)=\P(X_2(j)=1)\\
&=&\sum_{k=j+1}^{j+r+s}\sum_{n=1}^\infty\P(X_2(-(r+s)n+k)=1, I(-(r+s)n+k-1)=2,\\
\noalign{\vglue-3mm}
&&\hskip1.5in{} I(-(r+s)n+k)\ne2,\ldots,I(j-1)\ne2)\\
&=&\sum_{k=j+1}^{j+r+s}\sum_{n=1}^\infty\bigg(1-{1\over N}\bigg)^{(r+s)n+j-k}\P(X_2(-(r+s)n+k)=1,\\
\noalign{\vglue-3mm}
&&\hskip2.in{} I(-(r+s)n+k-1)=2)\\
&=&\sum_{k=j+1}^{j+r+s}\sum_{n=1}^\infty\bigg(1-{1\over N}\bigg)^{(r+s)n+j-k}{1\over N}\sum_{w=0}^1\sum_{z=0}^1\pi_{1,3}^{k-1}(w,z)p_{2w+z}(k),
\end{eqnarray*}
where the last equality follows by conditioning on $X_1(-(r+s)n+k-1)$ and $X_3(-(r+s)n+k-1)$; here $\pi_{1,3}^{j+r+s}:=\pi_{1,3}^j$
and $p_m(k):=p_m$ if $k-1$ (mod $r+s$) belongs to $\{r,\ldots,r+s-1\}$ and $p_m(k):=1/2$ otherwise.  Using the fact that
\begin{eqnarray}\label{geometric}
&&\sum_{j=0}^{r+s-1}\sum_{k=j+1}^{j+r+s}\sum_{n=1}^\infty\bigg(1-{1\over N}\bigg)^{(r+s)n+j-k}{1\over N}\alpha(k)\\
&&\quad{}=\sum_{j=0}^{r+s-1}\sum_{l=1}^{r+s}\sum_{n=1}^\infty\bigg(1-{1\over N}\bigg)^{(r+s)n-l}{1\over N}\alpha(j+l)\nonumber\\
&&\quad{}=\sum_{l=1}^{r+s}{(1-1/N)^{r+s-l}\over N[1-(1-1/N)^{r+s}]}\sum_{j=0}^{r+s-1}\alpha(j+l)
=\sum_{l=1}^{r+s}\alpha(l)\nonumber
\end{eqnarray}
if $\alpha$ is periodic with period $r+s$, this implies that
$$
{1\over r+s}\sum_{j=0}^{r+s-1}\pi_1^j(1)={r\over r+s}\,{1\over2}+{1\over r+s}\sum_{v=0}^{s-1}\sum_{w=0}^1\sum_{z=0}^1[\bm\pi\bm P_A^r\bm P_B^v]_{1,3}(w,z)p_{2w+z}.
$$
Therefore, (\ref{mu3}) follows from (\ref{mu2}).

Finally, (\ref{mu4}) follows by replacing the sum over $j$ in (\ref{geometric}) by the $j=r$ term, assuming $s=1$:
\begin{eqnarray*}
&&\sum_{k=r+1}^{2r+1}\sum_{n=1}^\infty\bigg(1-{1\over N}\bigg)^{(r+1)n+r-k}{1\over N}\alpha(k)\\
&&\quad{}=\sum_{l=1}^{r+1}\sum_{n=1}^\infty\bigg(1-{1\over N}\bigg)^{(r+1)n-l}{1\over N}\alpha(r+l)\\
&&\quad{}=\sum_{l=1}^{r+1}{(1-1/N)^{r+1-l}\over N[1-(1-1/N)^{r+1}]}\alpha(r+l)\\
&&\quad{}={(1-1/N)^r\over N[1-(1-1/N)^{r+1}]}\alpha(r+1)+\bigg(1-{(1-1/N)^r\over N[1-(1-1/N)^{r+1}]}\bigg){1\over2}
\end{eqnarray*}
if $\alpha(1)=\cdots=\alpha(r)=\alpha(r+2)=\cdots=\alpha(2r+1)=1/2$.  This implies that
$$
\pi_1^r(1)-\pi_1^r(0)={(1-1/N)^r\over N[1-(1-1/N)^{r+1}]}\sum_{w=0}^1\sum_{z=0}^1[\bm\pi\bm P_A^r]_{1,3}(w,z)(p_{2w+z}-q_{2w+z}),
$$
and, combined with (\ref{mu2}), this yields (\ref{mu4}).
\end{proof}

We conclude with an application of the SLLN.

Let us denote $\mu_B^N$, the mean profit per turn to the ensemble of $N$ players always playing game $B$, by $\mu_B^N(p_0,p_1,p_2,p_3)$ to emphasize its dependence on the parameter vector.  As shown in \cite{EL12b},
\begin{equation}\label{couple B}
\mu_B^N(p_0,p_1,p_2,p_3)=-\mu_B^N(q_3,q_2,q_1,q_0),
\end{equation}
where $q_m:=1-p_m$ for $m=0,1,2,3$.

Fix $r,s\ge1$.  Let us denote $\mu_{[r,s]}^N$ of Theorem~\ref{SLLN-thm} by $\mu_{[r,s]}^N(p_0,p_1,p_2,p_3)$.  A similar argument (see \cite{EL12d}) implies that
\begin{equation}\label{couple [r,s]}
\mu_{[r,s]}^N(p_0,p_1,p_2,p_3)=-\mu_{[r,s]}^N(q_3,q_2,q_1,q_0).
\end{equation}

We say the \textit{Parrondo effect} is present if $\mu_B^N\le0$ and $\mu_{[r,s]}^N>0$, whereas the \textit{anti-Parrondo} effect is present if $\mu_B^N\ge0$ and $\mu_{[r,s]}^N<0$.  Eqs.\ (\ref{couple B}) and (\ref{couple [r,s]}) imply that the Parrondo effect is present for the parameter vector $(p_0,p_1,p_2,p_3)$ if and only if the anti-Parrondo effect is present for the parameter vector $(q_3,q_2,q_1,q_0)$.  Since the transformation
$$
\Lambda(p_0,p_1,p_2,p_3):=(1-p_3,1-p_2,1-p_1,1-p_0)
$$
from $(0,1)^4$ to $(0,1)^4$ has Jacobian identically equal to $1$, it follows that the ``Parrondo region'' and the ``anti-Parrondo region'' have the same (four-dimensional) volume.

Similarly, if we restrict attention to parameter vectors $(p_0,p_1,p_2,p_3)$ with $p_1=p_2$, then the Parrondo region and the anti-Parrondo region have the same (three-dimensional) volume.

\section{The case $p_0=1$, $p_3=0$}\label{p0=1,p3=0}

\begin{theorem}\label{p0=1,p3=0-thm}
Let $p_0=1$, $p_1=p_2\in(1/2,1)$, and $p_3=0$.  Let $\mu^N_B$ (resp., $\mu_{[r,s]}^N$) denote the mean profit per turn to the ensemble of $N\ge3$ players always playing game $B$ (resp., repeatedly playing the nonrandom pattern $[r,s]$, where $r,s\ge1$).  Then $\mu^N_B=0$ for all even $N\ge4$, $\mu^N_B>0$ for all odd $N\ge3$, and $\mu_{[r,1]}^N>0$ for all $N\ge3$ and $r\ge1$.  In particular, the Parrondo effect is present for the nonrandom pattern $[r,s]$ if and only if $N$ is even, at least when $s=1$.
\end{theorem}

\begin{remark}
We expect that the condition $s=1$ is unnecessary for this result.
\end{remark}

\begin{proof}
The conclusions about game $B$ are from \cite{EL12c}.

Let $\pi_{1,3}^r$ be the $1,3$ two-dimensional marginal of $\bm\pi^r:=\bm\pi\bm P_A^r$ when the probability parameters are $1$, $p_1$, $p_1$, and $0$.  Here $\bm\pi$ is the unique stationary distribution of $\bm P_A^r\bm P_B^s$.  We apply Theorem~\ref{SLLN-thm} twice.  By (\ref{mu4}) and (\ref{1,3 symm}),
\begin{eqnarray*}
\mu_{[r,1]}^N&=&{N[1-(1-1/N)^{r+1}]\over(r+1)(1-1/N)^r}\{\pi_{1,3}^r(1,0)+\pi_{1,3}^r(1,1)-[\pi_{1,3}^r(0,0)+\pi_{1,3}^r(0,1)]\}\\
&=&{N[1-(1-1/N)^{r+1}]\over(r+1)(1-1/N)^r}[\pi_{1,3}^r(1,1)-\pi_{1,3}^r(0,0)].\nonumber
\end{eqnarray*}
By (\ref{mu2}), (\ref{1,3 symm}), and the preceding formula,
\begin{eqnarray*}
\mu_{[r,1]}^N&=&(r+1)^{-1}[\pi_{1,3}^r(0,0)(1)+2\pi_{1,3}^r(0,1)(2p_1-1)+\pi_{1,3}^r(1,1)(-1)]\\
&=&{2(2p_1-1)\over r+1}\pi_{1,3}^r(0,1)-{(1-1/N)^r\over N[1-(1-1/N)^{r+1}]}\mu_{[r,1]}^N.
\end{eqnarray*}
Therefore, 
$$
\mu_{[r,1]}^N={2(2p_1-1)\over r+1}\bigg(1+{(1-1/N)^r\over N[1-(1-1/N)^{r+1}]}\bigg)^{-1}\pi_{1,3}^r(0,1), 
$$
and this is positive by the irreducibility of the Markov chain with transition matrix $\bm P_r:=\bm P_B\bm P_A^r$ (Lemma \ref{ergodic}) and the assumption that $p_1>1/2$.
\end{proof}

\section{A spin system}\label{spin}

As shown in \cite{EL12c}, the discrete-time Markov chain for game $B$ converges in distribution, after rescaling its time parameter, to a \textit{spin system} on the one-dimensional integer lattice ${\bf Z}$.  Let us recall the limiting process as described by its generator.  Its state space is the product space
\begin{equation*}
\{0,1\}^{\bf Z}:=\{\bm x=(\ldots,x_{-2},x_{-1},x_0,x_1,x_2,\ldots): x_i\in\{0,1\}{\rm\ for\ all\ }i\in{\bf Z}\}.
\end{equation*}
We will usually refer to $x_i$ as the status (loser or winner, 0 or 1) of player $i$; occasionally, it will be convenient to refer to it as the spin at site $i$.  Let $m_i(\bm x):=2x_{i-1}+x_{i+1}$ as before but without the boundary conditions.  Also, let $\bm x^i$ be the element of $\{0,1\}^{\bf Z}$ equal to $\bm x$ except at the $i$th component;  for example, $\bm x^0:=(\ldots,x_{-2},x_{-1},1-x_0,x_1,x_2,\ldots)$.

The generator depends on the four probability parameters $p_0,p_1,p_2,p_3\in[0,1]$, and it has the form
\begin{eqnarray}\label{L_B}
(\mathscr{L}_Bf)(\bm x)&:=&\sum_{i\in{\bf Z}}c_i(\bm x)[f(\bm x^i)-f(\bm x)]
\end{eqnarray}
for functions $f$ depending on only finitely many components, where the \textit{flip rates} are given by
\begin{equation}\label{rates}
c_i(\bm x):=\begin{cases}p_{m_i(\bm x)}&\text{if $x_i=0$,}\\q_{m_i(\bm x)}&\text{if $x_i=1$,}
\end{cases}
\end{equation}
and $q_m:=1-p_m$ for $m=0,1,2,3$.  It can be shown that the functions depending on only finitely many components form a core for the generator of the Feller semigroup associated with the process.

For later use let us also define
\begin{equation}\label{L_A}
(\mathscr{L}_Af)(\bm x):=\sum_{i\in{\bf Z}}{1\over2}[f(\bm x^i)-f(\bm x)],
\end{equation}
which is just the special case of (\ref{L_B}) with $p_0=p_1=p_2=p_3=1/2$.

Next we would like to clarify the statement that this spin system is the limit in distribution of the $N$-player chain for game $B$ after an appropriate time change.  First, it is convenient to relabel the $N$ players.  Instead of labeling them from 1 to $N$, we label them  from $l_N$ to $r_N$, where
$$
l_N:=\begin{cases}-(N-1)/2&\text{if $N$ is odd,}\\
-N/2&\text{if $N$ is even,}
\end{cases}\;\;\text{and}\;\;
r_N:=\begin{cases}(N-1)/2&\text{if $N$ is odd,}\\
N/2-1&\text{if $N$ is even,}\\
\end{cases}
$$
with the understanding that players $l_N$ and $r_N$ are nearest neighbors.  The state space is
\begin{equation*}
\Sigma_N:=\{\bm x=(x_{l_N},\ldots,x_{-1},x_0,x_1,\ldots,x_{r_N}): x_i\in\{0,1\}{\rm\ for\ }i=l_N,\ldots,r_N\}.
\end{equation*}
(This is what we previously called $\Sigma$ but with the players relabeled.  To avoid confusion, we make the dependence on $N$ explicit in the notation.)  We also speed up time in the $N$-player model so that $N$ one-step transitions occur per unit of time.  The resulting discrete generator has the form
\begin{eqnarray*}
(\mathscr{L}_B^Nf)(\bm x)&:=&N\E[f(\bm X_N(1))-f(\bm x)\mid \bm X_N(0)=\bm x]
\end{eqnarray*}
where $x_{l_N-1}:=x_{r_N}$ and $x_{r_N+1}:=x_{l_N}$.
Consequently, if we define $\zeta_N:\Sigma_N\mapsto\{0,1\}^{\bf Z}$ by
\begin{equation}\label{zeta}
\zeta_N(x_{l_N},\ldots,x_{r_N}):=(\ldots,0,0,x_{l_N},\ldots,x_{r_N},0,0,\ldots),
\end{equation}
then $\mathscr{L}_B^N(f\circ\zeta_N)=(\mathscr{L}_Bf)\circ\zeta_N$ for all $\bm x\in\Sigma_N$ and $N\ge2K+4$, where $f(\bm x)$ depends on $\bm x$ only through the $2K+1$ components $x_i$, $-K\le i\le K$.

This shows that, if the spin system has a unique stationary distribution, then the unique stationary distribution of the $N$-player Markov chain (assumed ergodic in the sense of Lemma~1 of \cite{EL12c}), converges to it in the topology of weak convergence (essentially Proposition I.2.14 of Liggett \cite{L85}).  Let us assume that the spin system has a unique stationary distribution $\pi$, and let us denote the unique stationary distribution of the $N$-player Markov chain by $\pi^N$.  (We previously denoted the latter by $\bm\pi$ but now it is necessary to make the dependence on $N$ explicit.  We do not use boldface for $\pi^N$ or $\pi$ because it is no longer useful or possible, respectively, to think of them as row vectors.)  The above argument shows that $\pi^N\zeta_N^{-1}\Rightarrow\pi$.  Let us denote their $-1,1$ two-dimensional marginals by $(\pi^N)_{-1,1}$ and $\pi_{-1,1}$, so that $(\pi^N)_{-1,1}\Rightarrow\pi_{-1,1}$ and
\begin{equation*}
\sum_{w=0}^1\sum_{z=0}^1(\pi^N)_{-1,1}(w,z)p_{2w+z}\to\sum_{w=0}^1\sum_{z=0}^1\pi_{-1,1}(w,z)p_{2w+z}.
\end{equation*}
Hence $\mu_B^N$, the mean profit per turn to the ensemble of $N$ players always playing game $B$, converges as $N\to\infty$ to a limit that can be expressed in terms of the spin system. 

Under what conditions does the spin system have a unique stationary distribution (equivalently, a unique invariant probability measure)?  In \cite{EL12c} we gave sufficient conditions for the spin system to be \textit{ergodic}, which means not only that there is a unique stationary distribution $\pi$ but that the process at time $t$ converges in distribution to $\pi$ as $t\to\infty$, regardless of the initial distribution.

\begin{theorem}\label{ergodicity}
With $p_0,p_1,p_2,p_3\in[0,1]$, the spin system on ${\bf Z}$ with flip rates (\ref{rates}) is ergodic if at least one of the following four conditions is satisfied:

\emph{(a)} (basic estimate applies)
\begin{equation*}
\max(|p_0-p_1|,|p_2-p_3|)+\max(|p_0-p_2|,|p_1-p_3|)<1;
\end{equation*}

\emph{(b)} (attractiveness or repulsiveness applies)
\begin{equation*}
0<\min(p_0,p_3)\le \min(p_1,p_2)\le\max(p_1,p_2)\le \max(p_0,p_3)<1;
\end{equation*}

\emph{(c)} (coalescing duality applies)
\begin{equation*}
\;\;\max(p_1,p_2,p_3,p_1+p_2-p_3)-p_3<p_0/2<\min(p_1,p_2,p_3,p_1+p_2-p_3);
\end{equation*}

\emph{(d)} (annihilating duality applies) 
\begin{equation*}
p_0,p_1,p_2,p_3\in(2\overline{p}-1,2\overline{p})\cap(0,1),\quad \overline{p}:=(p_0+p_1+p_2+p_3)/4.
\end{equation*}
\end{theorem}

See \cite{EL12c} for further discussion.  The following result is immediate.

\begin{theorem}\label{lim mu^N}
Assume that $(p_0,p_1,p_2,p_3)$ is such that we can define $\mu_B^N=\mu_B^N(p_0,p_1,p_2,p_3)$ for all $N\ge3$ (this requires that the conditions for ergodicity in Lemma 1 of \cite{EL12c} are satisfied).  Assume also that the spin system on ${\bf Z}$ with flip rates (\ref{rates}) is ergodic (see Theorem~\ref{ergodicity} for sufficient conditions) with unique stationary distribution $\pi$.  Then $\lim_{N\to\infty}\mu_B^N=\mu_B$, where
$$
\mu_B:=\sum_{w=0}^1\sum_{z=0}^1\pi_{-1,1}(w,z)(p_{2w+z}-q_{2w+z}).
$$

Let $0<\gamma<1$ and assume that $(p_0,p_1,p_2,p_3)$ is such that we can define 
$$
\mu_{(\gamma,1-\gamma)}^N:=\mu_B^N(p_0(\gamma),p_1(\gamma),p_2(\gamma),p_3(\gamma)), 
$$
for all $N\ge3$, where
\begin{equation}\label{p_m()}
p_m(\gamma):=\gamma(1/2)+(1-\gamma)p_m,\qquad m=0,1,2,3,
\end{equation}
and that the spin system on ${\bf Z}$ with flip rates of the form (\ref{rates}) but with $(p_0,p_1,p_2,\linebreak p_3)$ replaced by $(p_0(\gamma),p_1(\gamma),p_2(\gamma),p_3(\gamma))$ is ergodic with unique stationary distribution $\pi^\gamma$.  Then $\lim_{N\to\infty}\mu_{(\gamma,1-\gamma)}^N=\mu_{(\gamma,1-\gamma)}$, where
\begin{equation}\label{mu_()}
\mu_{(\gamma,1-\gamma)}:=(1-\gamma)\sum_{w=0}^1\sum_{z=0}^1(\pi^\gamma)_{-1,1}(w,z)(p_{2w+z}-q_{2w+z}).
\end{equation}
\end{theorem}

We notice that condition (a) of Theorem \ref{ergodicity} holds with $(p_0,p_1,p_2,p_3)$ replaced by $(p_0(\gamma),p_1(\gamma),p_2(\gamma),p_3(\gamma))$ if
$$
\max(|p_0-p_1|,|p_2-p_3|)+\max(|p_0-p_2|,|p_1-p_3|)<1/(1-\gamma);
$$
this is automatic if $\gamma>1/2$.

The special case of Theorem \ref{lim mu^N} in which $\gamma=1/2$ was included in \cite{EL12c}.

\section{Convergence of means}\label{limit}

We turn to our main result, namely that $\lim_{N\to\infty}\mu_{[r,s]}^N$ exists under certain conditions.

\begin{theorem}
Fix $r,s\ge1$ and put $\gamma:=r/(r+s)$.  Assume that the spin system on ${\bf Z}$ with flip rates of the form (\ref{rates}) but with $(p_0,p_1,p_2,p_3)$ replaced by $(p_0(\gamma),p_1(\gamma),p_2(\gamma),p_3(\gamma))$ (see (\ref{p_m()})) is ergodic with unique stationary distribution $\pi^\gamma$.  Then $\lim_{N\to\infty}\mu_{[r,s]}^N=\mu_{(\gamma,1-\gamma)}$, where $\mu_{(\gamma,1-\gamma)}$ is as in (\ref{mu_()}).
\end{theorem}

\begin{proof}
Define $\zeta_N:\Sigma_N\mapsto\{0,1\}^{\bf Z}$ by (\ref{zeta}).  The main step is to show that the discrete generator $\mathscr{L}_{[r,s]}^N$, corresponding to the nonrandom pattern $[r,s]$ (and with $N$ games played per unit of time), satisfies
\begin{equation}\label{convergence-generators}
\mathscr{L}_{[r,s]}^N(f\circ\zeta_N)=[(r+s)^{-1}(r\mathscr{L}_A+s\mathscr{L}_B)f]\circ\zeta_N+O(N^{-1}),
\end{equation}
uniformly over $\Sigma_N$, for all $f$ depending on only finitely many components, where
where $\mathscr{L}_A$ and $\mathscr{L}_B$ are as in (\ref{L_A}) and (\ref{L_B}).
Because the result is nonintuitive and the proof is technical, we treat the case $r=s=1$ first.  We hope this slight redundancy will improve clarity.

In the case $r=s=1$, the discrete generator has the form
$$
(\mathscr{L}_{[1,1]}^Nf)(\bm x):={N\over2}\sum_{\bm z}[f(\bm z)-f(\bm x)](\bm P_A\bm P_B)(\bm x,\bm z).
$$
To evaluate this, we will need
\begin{eqnarray*}
P_A(\bm x,\bm y)&=&{1\over2}\delta(\bm x,\bm y)+{1\over2N}\sum_i\delta(\bm x^i,\bm y),\\
P_B(\bm y,\bm z)&=&{1\over N}\sum_j[1-c_j(\bm y)]\delta(\bm y,\bm z)+{1\over N}\sum_j c_j(\bm y)\delta(\bm y^j,\bm z),
\end{eqnarray*}
where $\delta(\bm x,\bm y)$ is the Kronecker delta, which equals 1 if $\bm x=\bm y$ and equals 0 otherwise; the sums over $i$ and $j$ range over $\{l_N,\ldots,r_N\}$; and $c_j(\bm y)$ is as in (\ref{rates}).  This tells us that
\begin{eqnarray*}
(\bm P_A\bm P_B)(\bm x,\bm z)&=&\sum_{\bm y}P_A(\bm x,\bm y)P_B(\bm y,\bm z)\\
&=&{1\over2N}\sum_j[1-c_j(\bm x)]\delta(\bm x,\bm z)+{1\over2N^2}\sum_i\sum_j[1-c_j(\bm x^i)]\delta(\bm x^i,\bm z)\\
&&\quad{}+{1\over2N}\sum_j c_j(\bm x)\delta(\bm x^j,\bm z)+{1\over2N^2}\sum_i\sum_j c_j(\bm x^i)\delta(\bm x^{ij},\bm z),
\end{eqnarray*}
where $\bm x^{ij}:=(\bm x^i)^j=(\bm x^j)^i$, so our discrete generator reduces to
\begin{eqnarray*}
(\mathscr{L}_{[1,1]}^Nf)(\bm x)&=&{1\over4N}\sum_i\sum_j[1-c_j(\bm x^i)][f(\bm x^i)-f(\bm x)]\\
&&\quad{}+{1\over4}\sum_jc_j(\bm x)[f(\bm x^j)-f(\bm x)]\\
&&\quad{}+{1\over4N}\sum_i\sum_jc_j(\bm x^i)[f(\bm x^{ij})-f(\bm x)].
\end{eqnarray*}

Now let us restrict attention to those functions $f$ that depend on only the coordinates $x_{-(K-1)},\ldots,x_{K-1}$ for some positive integer $K$.  Then we can write the discrete generator as
\begin{eqnarray*}
(\mathscr{L}_{[1,1]}^Nf)(\bm x)&=&{1\over4N}\sum_{|i|\le K}\sum_{|j|\le K}[1-c_j(\bm x^i)][f(\bm x^i)-f(\bm x)]\\
&&\quad{}+{1\over4N}\sum_{|i|\le K}\sum_{|j|>K}[1-c_j(\bm x^i)][f(\bm x^i)-f(\bm x)]\\
&&\quad{}+{1\over4}\sum_{|j|\le K}c_j(\bm x)[f(\bm x^j)-f(\bm x)]\\
&&\quad{}+{1\over4N}\sum_{|i|\le K}\sum_{|j|\le K}c_j(\bm x^i)[f(\bm x^{ij})-f(\bm x)]\\
&&\quad{}+{1\over4N}\sum_{|i|\le K}\sum_{|j|>K}c_j(\bm x^i)[f(\bm x^i)-f(\bm x)]\\
&&\quad{}+{1\over4N}\sum_{|i|>K}\sum_{|j|\le K}c_j(\bm x^i)[f(\bm x^j)-f(\bm x)].
\end{eqnarray*}
Let us refer to the six terms of the expression on the right as terms 1--6.  Terms 2 and 5 combine to give
\begin{eqnarray*}
{1\over4N}\sum_{|i|\le K}\sum_{|j|>K}[f(\bm x^i)-f(\bm x)]&=&{1\over4}\sum_{|i|\le K}{N-(2K+1)\over N}[f(\bm x^i)-f(\bm x)]\\
&=&{1\over4}\sum_{|i|\le K}[f(\bm x^i)-f(\bm x)]+O(N^{-1}).
\end{eqnarray*}
Term 3 is constant and term 6 simplifies to
\begin{eqnarray*}
{1\over4N}\sum_{|i|>K}\sum_{|j|\le K}c_j(\bm x^i)[f(\bm x^j)-f(\bm x)]&=&{1\over4N}\sum_{|i|>K}\sum_{|j|\le K}c_j(\bm x)[f(\bm x^j)-f(\bm x)]\\
&=&{N-(2K+1)\over4N}\sum_{|j|\le K}c_j(\bm x)[f(\bm x^j)-f(\bm x)]\\
&=&{1\over4}\sum_{|j|\le K}c_j(\bm x)[f(\bm x^j)-f(\bm x)]+O(N^{-1})
\end{eqnarray*}
because $c_j(\bm x^i)=c_j(\bm x)$ if $|i|>K$ and $|j|\le K$, with a possible exception when $|i|=K+1$, $|j|=K$, and $i$ and $j$ have the same sign, in which case $f(\bm x^j)-f(\bm x)=0$.  Finally, terms 1 and 4 are $O(N^{-1})$, so we conclude that
\begin{eqnarray*}
(\mathscr{L}_{[1,1]}^Nf)(\bm x)&=&\sum_{|j|\le K}\bigg({1\over4}+{1\over2}c_j(\bm x)\bigg)[f(\bm x^j)-f(\bm x)]+O(N^{-1})\\
&=&\sum_j\bigg({1\over4}+{1\over2}c_j(\bm x)\bigg)[f(\bm x^j)-f(\bm x)]+O(N^{-1})\\
&=&{1\over2}(\mathscr{L}_Af+\mathscr{L}_Bf)(\zeta_N(\bm x))+O(N^{-1}),
\end{eqnarray*}
which leads to (\ref{convergence-generators}) with $r=s=1$.

The general case should now be easier to follow.\footnote{By convention, $\prod_{i=1}^n a_ib$ equals $(a_1a_2\cdots a_n)b$, not $(a_1a_2\cdots a_n)b^n$.}  Given $r,s\ge1$, we evaluate
\begin{eqnarray*}
&&[\bm P_A^r\bm P_B^s](\bm x_0,\bm x_{r+s})\\
&&\;{}=\sum_{\bm x_1,\ldots,\bm x_{r+s-1}}\prod_{u=1}^rP_A(\bm x_{u-1},\bm x_u)\prod_{u=r+1}^{r+s}P_B(\bm x_{u-1},\bm x_u)\\
&&\;{}=\sum_{\bm x_1,\ldots,\bm x_{r+s-1}}\prod_{u=1}^r\bigg[{1\over2}\delta(\bm x_{u-1},\bm x_u)+{1\over2N}\sum_{i_u}\delta(\bm x_{u-1}^{i_u},\bm x_u)\bigg]\\
&&\quad\;\;{}\cdot\prod_{u=r+1}^{r+s}\bigg[{1\over N}\sum_{i_u}[1-c_{i_u}(\bm x_{u-1})]\delta(\bm x_{u-1},\bm x_u)+{1\over N}\sum_{i_u}c_{i_u}(\bm x_{u-1})\delta(\bm x_{u-1}^{i_u},\bm x_u)\bigg]\\
&&\;{}={1\over2^r}\sum_{A\subset\{1,\ldots,r\}}\sum_{B\subset\{r+1,\ldots,r+s\}}\sum_{\bm x_1,\ldots,\bm x_{r+s-1}}\prod_{u\in A^c}\delta(\bm x_{u-1},\bm x_u)\\
&&\quad\;\;{}\cdot\prod_{u\in A}\bigg[{1\over N}\sum_{i_u}\delta(\bm x_{u-1}^{i_u},\bm x_u)\bigg]\prod_{u\in B^c}\bigg[{1\over N}\sum_{i_u}[1-c_{i_u}(\bm x_{u-1})]\delta(\bm x_{u-1},\bm x_u)\bigg]\\
&&\quad\;\;{}\cdot\prod_{u\in B}\bigg[{1\over N}\sum_{i_u}c_{i_u}(\bm x_{u-1})\delta(\bm x_{u-1}^{i_u},\bm x_u)\bigg]\\
&&\;{}={1\over2^r}\sum_{A\subset\{1,\ldots,r\}}{1\over N^{|A|+s}}\sum_{B\subset\{r+1,\ldots,r+s\}}\sum_{\bm x_1,\ldots,\bm x_{r+s-1}}\\
&&\quad\;\;{}\cdot\sum_{i_u:u\in A}\sum_{i_u:u\in B^c}\sum_{i_u:u\in B}\prod_{v\in B^c}[1-c_{i_v}(\bm x_{v-1})]\prod_{v\in B}c_{i_v}(\bm x_{v-1})\\
&&\quad\;\;{}\cdot\prod_{v\in A^c\cup B^c}\delta(\bm x_{v-1},\bm x_v)\prod_{v\in A\cup B}\delta(\bm x_{v-1}^{i_v},\bm x_v)\\
&&\;{}={1\over2^r}\sum_{A\subset\{1,\ldots,r\}}{1\over N^{|A|+s}}\sum_{B\subset\{r+1,\ldots,r+s\}}\sum_{i_u:u\in A\cup \{r+1,\ldots,r+s\}}\\
&&\quad\;\;{}\cdot\prod_{v\in B^c}[1-c_{i_v}(\bm x_0^{\{i_w:w\in A\cup B, w<v\}})]\prod_{v\in B}c_{i_v}(\bm x_0^{\{i_w:w\in A\cup B, w<v\}})\\
&&\quad\;\;{}\cdot\delta(\bm x_0^{\{i_v:v\in A\cup B\}},\bm x_{r+s});
\end{eqnarray*}
here $A^c:=\{1,\ldots,r\}-A$ and $B^c:=\{r+1,\ldots,r+s\}-B$; also, $\bm x_0^{\{i_v:v\in A\cup B\}}$, for example, denotes $\bm x_0$ with the spin flipped at each site $i_v$ with $v\in A\cup B$;  these site labels are not necessarily distinct, so if there are multiple flips at a single site, only the parity of the number of flips is relevant.

With $f(\bm x)$ depending only on $x_{-(K-1)},\ldots,x_{K-1}$ for some positive integer $K$, this leads to
\begin{eqnarray}\label{Lf}
(\mathscr{L}_{[r,s]}^N f)(\bm x_0)
&=&{N\over r+s}\sum_{\bm x_{r+s}}[f(\bm x_{r+s})-f(\bm x_0)][\bm P_A^r\bm P_B^s](\bm x_0,\bm x_{r+s})\nonumber\\
&=&{N\over r+s}\,{1\over2^r}\sum_{A\subset\{1,\ldots,r\}}{1\over N^{|A|+s}}\sum_{B\subset\{r+1,\ldots,r+s\}}\sum_{i_u:u\in A\cup\{r+1,\ldots,r+s\}}\\
&&\;\;{}\cdot\prod_{v\in B^c}[1-c_{i_v}(\bm x_0^{\{i_w:w\in A\cup B, w<v\}})]\prod_{v\in B}c_{i_v}(\bm x_0^{\{i_w:w\in A\cup B, w<v\}})\nonumber\\
&&\;\;{}\cdot[f(\bm x_0^{\{i_v:v\in A\cup B\}})-f(\bm x_0)].\nonumber
\end{eqnarray}
Now, with error at most $O(N^{-1})$, we can replace $\sum_{i_u:u\in A\cup\{r+1,\ldots,r+s\}}$ by
\begin{eqnarray}\label{A,B}
&&\sum_{u\in A}\sum_{|i_u|\le K}\sum_{|i_z|>K:z\in A\cup\{r+1,\ldots,r+s\},z\ne u}\\
&&\quad{}+\sum_{u\in B}\sum_{|i_u|\le K}\sum_{|i_z|>K:z\in A\cup\{r+1,\ldots,r+s\},z\ne u}.\nonumber
\end{eqnarray}
The justification is that each sum $\sum_{i_u}$ can be written as $\sum_{|i_u|\le K}+\sum_{|i_u|>K}$, resulting in $2^{|A|+s}$ multiple sums.  But each of those multiple sums with two or more sums of the form $\sum_{|i_u|\le K}$ contributes $O(N^{-1})$, and those with no sums of the form $\sum_{|i_u|\le K}$, where $u\in A\cup B$, are 0.  

But before evaluating the result, let us make one more simplification.  We replace the argument of $c_{i_v}$ in (\ref{Lf}) by just $\bm x_0$.  Here the justification is that $c_{i_v}(\bm x_0^{\{i_w:w\in A\cup B, w<v\}})=c_{i_v}(\bm x_0)$ for all but at most $3(r+s)$ of the $N$ possible values of $i_v$ (namely, $i_w-1,i_w,i_w+1$ for $w=1,2,\ldots,r+s$), hence the approximation introduces an error that is $O(N^{-1})$.  The result is that $(\mathscr{L}_{[r,s]}^N f)(\bm x_0)$ can be written as the sum of two terms corresponding to the two multiple sums in (\ref{A,B}), plus $O(N^{-1})$.  

The term corresponding to the first multiple sum in (\ref{A,B}) is, up to $O(N^{-1})$,
\begin{eqnarray}\label{A term}
&&{N\over r+s}\,{1\over2^r}\sum_{A\subset\{1,\ldots,r\}}{1\over N^{|A|+s}}\sum_{B\subset\{r+1,\ldots,r+s\}}\sum_{u\in A}\sum_{|i_u|\le K}[f(\bm x_0^{i_u})-f(\bm x_0)]\\
&&\;{}\cdot\sum_{|i_z|>K:z\in A\cup\{r+1,\ldots,r+s\},z\ne u}\prod_{v\in B^c}[1-c_{i_v}(\bm x_0)]\prod_{v\in B}c_{i_v}(\bm x_0).\nonumber
\end{eqnarray}
Now since
\begin{eqnarray*}
\sum_{B\subset\{r+1,\ldots,r+s\}}\prod_{v\in B^c}[1-c_{i_v}(\bm x_0)]\prod_{v\in B}c_{i_v}(\bm x_0)=\prod_{v=r+1}^{r+s}[1-c_{i_v}(\bm x_0)+c_{i_v}(\bm x_0)]=1
\end{eqnarray*}
and since
$$
{1\over2^r}\sum_{A\subset\{1,\ldots,r\}}|A|=\sum_{k=0}^rk{r\choose k}2^{-r}={r\over2},
$$
(\ref{A term}) becomes, up to $O(N^{-1})$,
$$
{r\over r+s}\sum_{|i|\le K}{1\over2}[f(\bm x_0^i)-f(\bm x_0)]={r\over r+s}(\mathscr{L}_A f)(\zeta_N(\bm x_0)).
$$

The term corresponding to the second multiple sum in (\ref{A,B}) is, up to $O(N^{-1})$,
\begin{eqnarray}\label{B term}
&&{N\over r+s}\,{1\over2^r}\sum_{A\subset\{1,\ldots,r\}}{1\over N^{|A|+s}}\sum_{B\subset\{r+1,\ldots,r+s\}}\sum_{u\in B}\sum_{|i_u|\le K}[f(\bm x_0^{i_u})-f(\bm x_0)]\\
&&\;\;{}\cdot\sum_{|i_z|>K:z\in A\cup\{r+1,\ldots,r+s\},z\ne u}\prod_{v\in B^c}[1-c_{i_v}(\bm x_0)]\prod_{v\in B}c_{i_v}(\bm x_0).\nonumber
\end{eqnarray}
Now
$$
\sum_{B\subset\{r+1,\ldots,r+s\}}\sum_{u\in B}=\sum_{u=r+1}^{r+s}\sum_{B\subset\{r+1,\ldots,r+s\}:u\in B},
$$
so (\ref{B term}) becomes
\begin{eqnarray*}
&&{N\over r+s}\,{1\over2^r}\sum_{A\subset\{1,\ldots,r\}}{1\over N^{|A|+s}}\sum_{u=r+1}^{r+s}\sum_{|i_u|\le K}c_{i_u}(\bm x_0)[f(\bm x_0^{i_u})-f(\bm x_0)]\nonumber\\
&&\;\;{}\cdot\sum_{|i_z|>K:z\in A\cup \{r+1,\ldots,r+s\},z\ne u}\sum_{B\subset\{r+1,\ldots,r+s\}:u\in B}\\
&&\;\;{}\cdot\prod_{v\in B^c}[1-c_{i_v}(\bm x_0)]\prod_{v\in B-\{u\}}c_{i_v}(\bm x_0)\\
&&{}={1\over r+s}\sum_{u=r+1}^{r+s}\sum_{|i_u|\le K}c_{i_u}(\bm x_0)[f(\bm x_0^{i_u})-f(\bm x_0)]+O(N^{-1})\\
&&{}={s\over r+s}(\mathscr{L}_B f)(\zeta_N(\bm x_0))+O(N^{-1}).
\end{eqnarray*}

Let us replace $f$ by $f\circ\zeta_N$, where $f\in C(\{0,1\}^{{\bm Z}})$ and $f(\bm x)$ depends on only the components $x_{-(K-1)},\ldots,x_{K-1}$.
We conclude that (\ref{convergence-generators}) holds,
uniformly over $\Sigma_N$, which ensures that the unique stationary distribution $\bm\pi^N$ of $\bm P_A^r\bm P_B^s$ converges weakly to the unique stationary distribution $\pi^{r/(r+s)}$ of the spin system with generator $\mathscr{L}_B$ but with $(p_0,p_1, p_2,p_3)$ replaced by
$(p_0(\gamma), p_1(\gamma),\linebreak p_2(\gamma),p_3(\gamma))$, where $\gamma:=r/(r+s)$, provided ergodicity holds for the limiting spin system.

The mean profit per turn to the ensemble of $N$ players playing the nonrandom periodic pattern $A^rB^s$ is, according to Theorem \ref{SLLN-thm},
\begin{equation}\label{mean}
\mu_{[r,s]}^N={1\over r+s}\sum_{v=0}^{s-1}\sum_{\bm x\in\Sigma}[\bm\pi^N\bm P_A^r\bm P_B^v](\bm x){1\over N}\sum_{i=1}^N [p_{m_i(\bm x)}-q_{m_i(\bm x)}].
\end{equation}
Now term $v$ of the sum in (\ref{mean}) can be expressed as
\begin{eqnarray*}
&&\sum_{\bm x_0,\bm x}\pi^N(\bm x_0)[\bm P_A^r\bm P_B^v](\bm x_0,\bm x){1\over N}\sum_l[p_{m_l(\bm x)}-q_{m_l(\bm x)}]\\
&=&{1\over2^r}\sum_{\bm x_0}\pi^N(\bm x_0)\sum_{A\subset\{1,\ldots,r\}}{1\over N^{|A|+v}}\sum_{B\subset\{r+1,\ldots,r+v\}}\sum_{i_u:u\in A\cup \{r+1,\ldots,r+v\}}\\
&&\;\;{}\cdot\prod_{w\in B^c}[1-c_{i_w}(\bm x_0^{\{i_z:z\in A\cup B, z<w\}})]\prod_{w\in B}c_{i_w}(\bm x_0^{\{i_z:z\in A\cup B, z<w\}})\\
&&\;\;{}\cdot {1\over N}\sum_l[p_{m_l(\bm x_0^{\{i_w:w\in A\cup B\}})}-q_{m_l(\bm x_0^{\{i_w:w\in A\cup B\}})}]\\
&=&{1\over2^r}\sum_{\bm x_0}\pi^N(\bm x_0)\sum_{A\subset\{1,\ldots,r\}}{1\over N^{|A|+v}}\sum_{B\subset\{r+1,\ldots,r+v\}}\sum_{i_u:u\in A\cup \{r+1,\ldots,r+v\}}\\
&&\;\;{}\cdot\prod_{w\in B^c}[1-c_{i_w}(\bm x_0)]\prod_{w\in B}c_{i_w}(\bm x_0){1\over N}\sum_l[p_{m_l(\bm x_0)}-q_{m_l(\bm x_0)}]+O(N^{-1})\\
&=&{1\over N}\sum_{\bm x_0}\pi^N(\bm x_0)\sum_l[p_{m_l(\bm x_0)}-q_{m_l(\bm x_0)}]+O(N^{-1})\\
&=&\sum_{w=0}^1\sum_{z=0}^1(\pi^N)_{-1,1}(w,z)(p_{2w+z}-q_{2w+z})+O(N^{-1})\\
&=&\sum_{w=0}^1\sum_{z=0}^1(\pi^{r/(r+s)})_{-1,1}(w,z)(p_{2w+z}-q_{2w+z})+o(1).
\end{eqnarray*}
Hence, using (\ref{p_m()}) with $\gamma:=r/(r+s)$, we have
$$
\mu_{[r,s]}^N\to(1-\gamma)\sum_{w=0}^1\sum_{z=0}^1(\pi^\gamma)_{-1,1}(w,z)(p_{2w+z}-q_{2w+z})=\mu_{(\gamma,1-\gamma)},
$$
as required.
\end{proof}


\begin{thebibliography}{00}

\bibitem{A10}\textsc{Abbott}, D. (2010) Asymmetry and disorder: A decade of  Parrondo's paradox. \textit{Fluct. Noise Lett.} \textbf{9} (1) 129--156.

\bibitem{EL09}\textsc{Ethier}, S. N. and \textsc{Lee}, J. (2009) Limit theorems for Parrondo's paradox. \textit{Electron. J. Probab.} {\bf 14} (62) 1827--1862.  \url{http://arxiv.org/abs/0902.2368}.

\bibitem{EL12a}\textsc{Ethier}, S. N. and \textsc{Lee}, J. (2012) Parrondo's paradox via redistribution of wealth. \textit{Electron. J. Probab.} \textbf{17} (20) 1--21.  \url{http://arxiv.org/abs/1109.4454}.

\bibitem{EL12b}\textsc{Ethier}, S. N. and \textsc{Lee}, J. (2012) Parrondo games with spatial dependence. \textit{Fluct. Noise Lett.} \textbf{11} (2), to appear.  \url{http://arxiv.org/abs/1202.2609}.

\bibitem{EL12c}\textsc{Ethier}, S. N. and \textsc{Lee}, J. (2012) Parrondo games with spatial dependence and a related spin system.  \url{http://arxiv.org/abs/1203.0818}.

\bibitem{EL12d}\textsc{Ethier}, S. N. and \textsc{Lee}, J. (2012) Parrondo games with spatial dependence, II.  \url{http://arxiv.org/abs/1206.6195}.

\bibitem{HA02}\textsc{Harmer}, G. P. and \textsc{Abbott}, D. (2002) A review of Parrondo's paradox. \textit{Fluct. Noise Lett.} \textbf{2} (2) R71--R107.

\bibitem{L85}\textsc{Liggett}, T. M. (1985) \textit{Interacting Particle Systems}. Springer-Verlag, New York.

\bibitem{MR03}\textsc{Mihailovi\'c}, Z. and \textsc{Rajkovi\'c}, M. (2003) One dimensional asynchronous cooperative Parrondo's games. \textit{Fluct. Noise Lett.} \textbf{3} (4) L389--L398.  \url{http://arxiv.org/abs/cond-mat/0307193}.

\bibitem{T01}\textsc{Toral}, R. (2001) Cooperative Parrondo games. \textit{Fluct. Noise Lett.} \textbf{1} (1) L7--L12.  \url{http://arxiv.org/abs/cond-mat/0101435}.

\bibitem{T02}\textsc{Toral}, R. (2002) Capital redistribution brings wealth by Parrondo's paradox. \textit{Fluct. Noise Lett.} \textbf{2} (4) L305--L311. \url{http://arxiv.org/abs/cond-mat/0206385}.


\end{thebibliography}
\end{document}